\documentclass[sn-mathphys,Numbered]{sn-jnl}

\usepackage{graphicx}%
\usepackage{multirow}%
\usepackage{amsmath,amssymb,amsfonts}%
\usepackage{amsthm}%
\usepackage{amscd}
\usepackage{mathrsfs}%
\usepackage[title]{appendix}%
\usepackage{xcolor}%
\usepackage{textcomp}%
\usepackage{manyfoot}%
\usepackage{booktabs}%
\usepackage{algorithm}%
\usepackage{algorithmicx}%
\usepackage{algpseudocode}%
\usepackage{listings}%

\theoremstyle{thmstyleone}%
\newtheorem{theorem}{Theorem}
\newtheorem{proposition}[theorem]{Proposition}%
\newtheorem{corollary}[theorem]{Corollary}
\newtheorem{lemma}[theorem]{Lemma}

\theoremstyle{thmstyletwo}%

\theoremstyle{thmstylethree}%
\newtheorem{dfn}{Definition}%

\begin{document}

\title[Article Title]{Novel definition and quantitative analysis of branch structure with topological data analysis}


\author*[1]{\fnm{Haruhisa} \sur{Oda}}\email{haruhisa-oda0722@g.ecc.u-tokyo.ac.jp}
\author[2]{\fnm{Mayuko} \sur{Kida}}
\author[3,4]{\fnm{Yoichi} \sur{Nakata}}
\author[2]{\fnm{Hiroki} \sur{Kurihara}}

\affil*[1]{\orgdiv{Faculty of Medicine}, \orgname{The University of Tokyo}, \orgaddress{\street{Hongo}, \city{Bunkyo-ku}, \postcode{1130033}, \state{Tokyo}, \country{Japan}}}
\affil[2]{\orgdiv{Department of Physiological Chemistry and Metabolism, Graduate School of Medicine}, \orgname{The University of Tokyo}, \orgaddress{\street{Hongo}, \city{Bunkyo-ku}, \postcode{1130033}, \state{Tokyo}, \country{Japan}}}
\affil[3]{\orgdiv{Isotope Science Center}, \orgname{The University of Tokyo}, \orgaddress{\street{Yayoi}, \city{Bunkyo-ku}, \postcode{1130032}, \state{Tokyo}, \country{Japan}}}
\affil[4]{\orgdiv{Arithmer Inc.}, \orgname{R\&D Headquarters}, \orgaddress{\street{Hongo}, \city{Bunkyo-ku}, \postcode{1130033}, \state{Tokyo}, \country{Japan}}}


\abstract{While branching network structures abound in nature, their objective analysis is more difficult than expected because existing quantitative methods often rely on the subjective judgment of branch structures. This problem is particularly pronounced when dealing with images comprising discrete particles. Here we propose an objective framework for quantitative analysis of branching networks by introducing the mathematical definitions for internal and external structures based on topological data analysis, specifically, persistent homology. We compare persistence diagrams constructed from images with and without plots on the convex hull. The unchanged points in the two diagrams are the internal structures and the difference between the two diagrams is the external structures. We construct a mathematical theory for our method and show that the internal structures have a monotonicity relationship with respect to the plots on the convex hull, while the external structures do not. This is the phenomenon related to the resolution of the image. Our method can be applied to a wide range of branch structures in biology, enabling objective analysis of numbers, spatial distributions, sizes, and more. Additionally, our method has the potential to be combined with other tools in topological data analysis, such as the generalized persistence landscape.}




\maketitle
\section*{Acknowledgements}
This work was supported in part by JSPS KAKENHI Grant Numbers 19H01048 and 22H04991 (to H.K.).

\section{Introduction}
Branching network patterns are pervasive in nature, exhibiting their intricate forms in diverse contexts. While a range of quantitative approaches has been applied to study the branching network properties\cite{intro1,intro2}, the recognition of branching structures has often relied on subjective and intuitive judgments. This challenge is particularly pronounced when dealing with original images comprising discrete and punctate elements, as exemplified by microscopic images of cellular nuclei composing blood and lymphatic vessels. The absence of an objective framework for definitively discerning and quantifying branch structures has hindered the advancement of rigorous numerical analyses. In the present study, we give an objective framework that provides a mathematical definition for branch structures based on topological data analysis (TDA), particularly, persistent homology. We further develop a method to extract information from branching structures in a way that can be used for quantitative analysis.

Here, we should mention the term “branch”
, which is often used for multiple meanings in biological settings.
In a broad sense, it refers to both the protruding segments and the internal network structures. In a narrow sense, it only means the protruding segments. In this manuscript, we will define the internal and external structures that basically extract information on the internal network structures and the protruding segments respectively.
Mathematically formal definitions of these structures will appear in the Theory section.

Plotting points on the boundary of the convex hull of the image plays an important role in this method.
(In this paper, in order to avoid writing `` the boundary of the convex hull" many times, we will just say ``the convex hull".)
These points can be regarded as parameters of this method. In the ``Theory'' section, we will mathematically explore how external and internal structures behave under these parameters. We will see that internal structures have some monotonicity relationship with respect to points on the convex hull, while external structures do not. We give some explanations for these results. In the ``Results'' section, we will apply our method to real images of lymph vessels to see how it works. We will show some possible ways of quantitative analysis including the use of persistence landscape. We will also deal with the images of neurons and blood vessels to see that our method can be applied to a wide range of branch structures that appear in biology. In the ``Discussion'' section, we compare our method with other methods dealing with some branch-like structures \cite{placentaarea,imageJskeleton}. We will see that these methods do not directly give an answer to the problem of what branch is. Also, for future applications to a broader range of image settings, we generalize persistence landscape\cite{bubenik1,bubenik2} using 2 parameters and show that there are some differentiability structures with respect to the 2 parameters.
The calculation of persistent homology is conducted using HomCloud (3.4.1)\cite{obayashi,obayashi2}. The calculation of the persistence landscape was conducted using the program code by Katherine Benjamin available at https://gitlab.com/kfbenjamin/pysistence-landscapes based on Bubenik(2017)\cite{bubenik4}. The programs are written in Python (version 3.7.9 and version 3.10.9). Note that image processing before the persistent homological analysis is not the focus of this paper. Therefore, we use binarized images which have already gone through several image processing steps including manual noise reduction as our data.

\section{Background}
Here, we explain briefly the concepts of persistent homology and persistence landscape. We only consider homology with coefficients in a field $\mathbb{K}$. Suppose that we are given a point cloud $X$. From $X$, we create some simplicial complex. In this paper, we use the \v{C}ech (alpha) complex. Let $C(r)$ be the \v{C}ech complex at the radius $r$. Suppose that $C(r)$ changes at
the radii
$r_1<\cdots<r_i<\cdots$. We define $C^i$ to be $C(r_i)$. We denote by $C^i_n$ the n-chains of $C^i$. From the definition of the \v{C}ech complex, we have $C^i_n\subset C^{i+1}_n$ for all $i,n$. We get the following commutative diagram with the horizontal arrows being the boundary maps ($\partial$) and the vertical arrows being the inclusion maps ($\iota$) induced by the above inclusion relationship.
\[
\begin{CD}
@. @VVV @VVV @VVV @. \\
@>>> C^{i}_{n+1} @>{\partial}>> C^{i}_{n} @>{\partial}>> C^{i}_{n-1} @>>> \\
@. @V{\iota}VV @V{\iota}VV @V{\iota}VV @. \\
@>>> C^{i+1}_{n+1} @>{\partial}>> C^{i+1}_{n} @>{\partial}>> C^{i+1}_{n-1} @>>> \\
@. @VVV @VVV @VVV @. \\
\end{CD}
\]
From the horizontal part, we can define
the n-th homology
$H_n(C^i)$. We define $H_n(C)$ to be the direct sum $\displaystyle\bigoplus_i H_n(C^i)$. We introduce a grading to $H_n(C)$ by defining $x\cdot: H_n(C^i)\rightarrow H_n(C^{i+1})$ as $x\cdot m:=\iota_*(m)$ for $m\in H_n(C^i)$. Here, $\iota_*$ is a homomorphism between homology groups induced by inclusion maps. This grading makes $H_n(C)$ into a $\mathbb{K}[x]$-graded module. We decompose this module by structure theorem\cite{zomorodian1}. The result of this decomposition is visualized by barcodes\cite{collins1,ghrist1}, the collection of bars. The left ends of the bars are called the births and the right ends of the bars are called the deaths. Persistence diagram\cite{cohen-steiner1} is a 2-dimensional representation of a barcode. For each bar in a barcode, we plot a point whose $x$-coordinate is the birth and whose $y$-coordinate is the death of the bar. We regard the barcode as the set of intervals $\{I_j\}$. Persistence landscape is the set of functions $\{\lambda_k(t)\}_{k=1,2,\ldots}$ defined as $\lambda_k(t):=\sup(u\ge 0| \#\{j|[t-u,t+u]\subset I_j\}\ge k)$.

\section{Method}
First, we explain the problem posed in the Introduction in more detail. In Frye et al. (2018) \cite{intro1}, they analyze the lymphatic vessel using criteria such as vessel diameter and branchpoints. These analyses are possible because they are dealing with structures that ``look like'' branches. Applying these criteria for the structures that do not look like branches (e.g. Fig.~\ref{sampleimages}(a)) will be misleading. Thus, we need to decide whether our structure is branch-like or not before we use the criteria in Frye et al. \cite{intro1}, but this decision is subjective. This also interrupts the comparison between a broader range of vessel structures such as Fig.~\ref{sampleimages}(a) and (b).
\begin{figure}[ht]
\centering
\includegraphics[width=\linewidth]{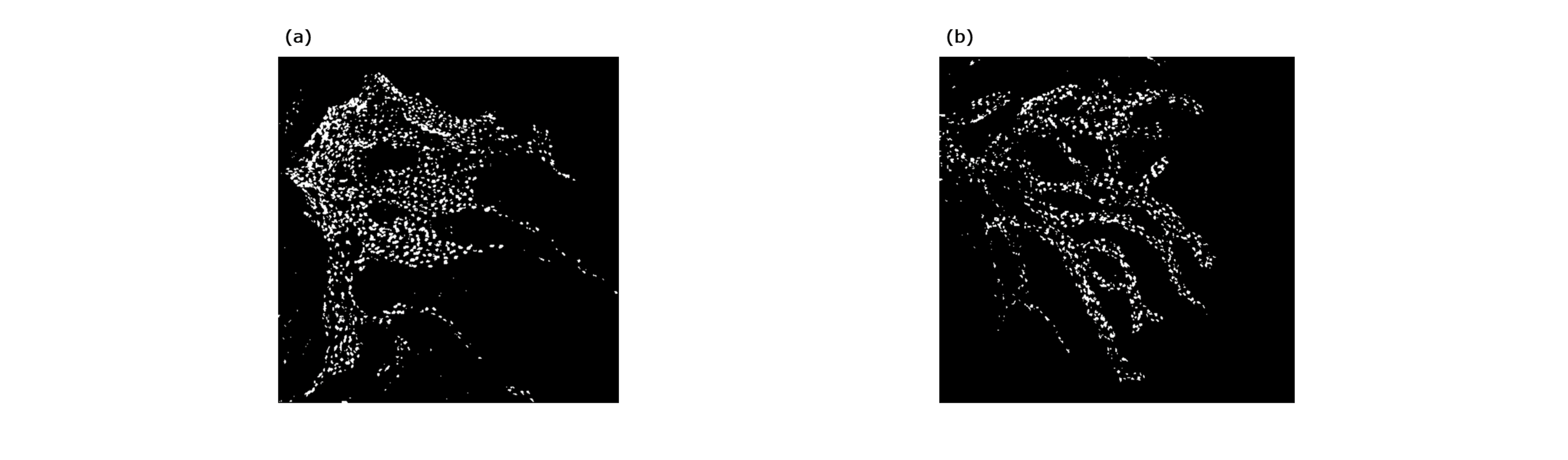}
\caption{(a,b) Examples of the images of lymph vessels that we will analyze in this paper}
\label{sampleimages}
\end{figure}
Since vessel structures have external parts and internal parts as we mentioned in the Introduction, one good way out of this problem is to separately analyze the external and internal parts. In Suchting et al. (2007) \cite{intro2}, they divide the blood vessels into 2 parts: vascular front (external parts) and vascular plexus (internal parts). In their images, the vessels form connected components, which enables the simple distinction of internal/external structures. On the other hand, our images are made of discrete cells. Thus, it is necessary to define the internal and external structures objectively. Now, we explain the motivations for persistent homological definitions of the internal/external structures and the method to extract the information from our images.

\begin{figure}[ht]
\centering
\includegraphics[width=\linewidth]{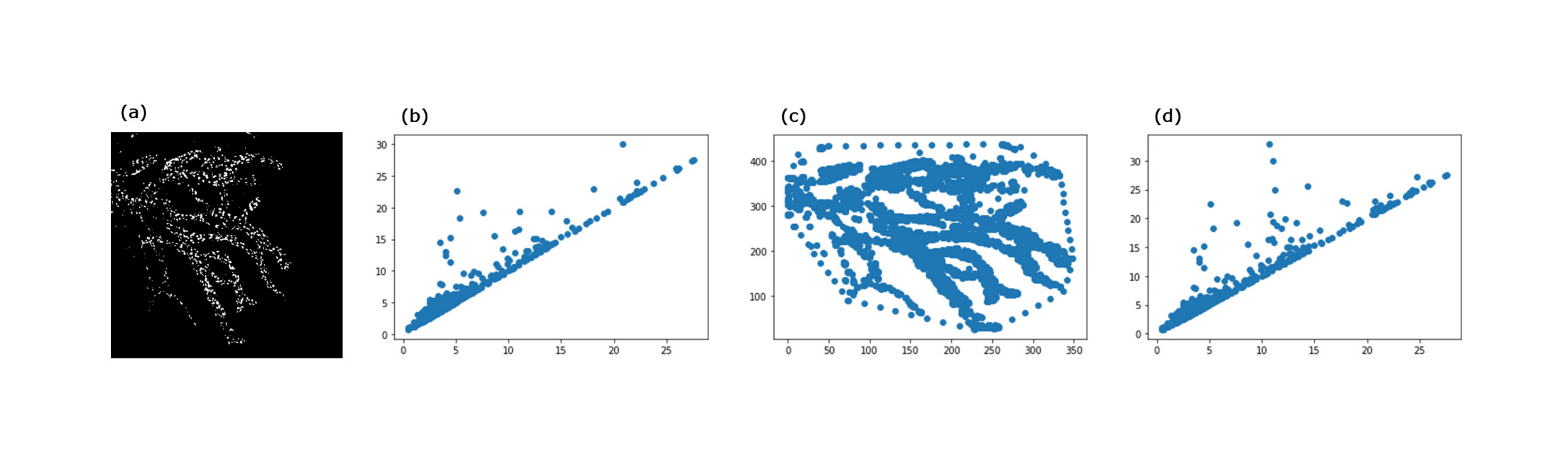}
\caption{(a) The image of lymph vessels. We use the white points as a point cloud. (b) The persistence diagram calculated from the original image of lymph vessels. (c) The points on the convex hull of the image and the original points. (d) The persistence diagram calculated from the points on the convex hull and the points of the original image}
\label{figures}
\end{figure}
From now on, we use 1-dimensional persistence diagrams unless otherwise stated, because what we want to detect is a ring structure in 2D images.
We use the white points in the original image (Fig.~\ref{figures}(a)) as a point cloud. We construct the persistence diagram from this image to get Fig.~\ref{figures}(b). Then we plot some points on the convex hull of the image (Fig.~\ref{figures}(c)). These points on the convex hull make the external structures
(based on our intuition, not defined yet) into ring structures that can be detected by persistent homology. We construct the persistence diagram from this new point cloud to get Fig.~\ref{figures}(d). We consider that internal structures should not be affected by the points on the convex hull, while external structures should be detected after adding points on the convex hull.
Therefore, we take the intersection of two persistence diagrams as the internal structures. We choose points in the latter diagram that are not in the former diagram and consider that they are the external structures. We can change the results by changing the number of points on the convex hull. If we want to compare different images, we should fix the interval of the plots. The theory in the next section is about how the external and internal structures behave when we change the plots on the convex hull.
\section{Theory}
First, we give a formal definition of internal/external structures. Let $PD_1(.)$ denote the 1-dimensional persistence diagram constructed from some point cloud.
\begin{dfn}
Let $X$ be a point cloud and $U$ be a set of points on its convex hull. We define the internal structure as $PD_1(X)\cap PD_1(X\cup U)$ and the external structure as $PD_1(X\cup U)\backslash PD_1(X)$.
\end{dfn}
Intuitively, the more points we plot on the convex hull, the fewer internal structures and the more external structures we will detect. We will see that the former is true but the latter does not hold. This statement, as it stands, is not mathematically sound, so we first make the mathematical meaning of this statement clear.

Let $X$ be an initial point cloud (the input image in our settings) and $U_1, U_2$ be sets of points on the convex hull. Suppose that there is an inclusion relationship between $U_1$ and $U_2$, namely, $\emptyset\subset U_1\subset U_2$. In the following arguments, we will change the radius of the complex used for persistent homology and the plots on the convex hull. When we want to clarify which we are talking about, we will use the notation like ``at $r=r_1$'' for radius and ``at $U=U_1$'' for plots on the convex hull instead of ``at $r_1$'' or ``at $U_1$''.

\subsection{Internal structures}
We focus on the internal structures. As we defined earlier, at $U=U_1$ and $U_2$, the detected internal structures are $PD_1(X)\cap PD_1(X\cup U_1)$ and $PD_1(X)\cap PD_1(X\cup U_2)$, respectively. The following definition explains how we relate loops in the images to the points in persistence diagrams. Here, given the points $a_1,\ldots,a_n$, the loop $[a_1,\ldots,a_n]$ means $\langle a_1,a_2\rangle+\cdots+\langle a_n,a_1\rangle$, where $\langle u,v\rangle$ represents a 1-simplex. For a loop $L$, we denote by $[L]$ the equivalence class of $L$.
\begin{dfn}
We say that a loop $L$ can represent a bar $[b,d]$ if $[L]$ can serve as a basis of the homology group of the bar at $r=b$.
\end{dfn}
We give an example in order to make clear the meaning of the following expression: a loop's equivalence class can serve as a basis of the homology group at the birth of a bar. Figure~\ref{barcodechoice} shows an image of a symmetric point cloud. If we calculate persistent homology from the image, we get the barcode shown in Fig.~\ref{barcodechoice}. The equivalence class of the loop constructed by all the points in the point cloud can serve as a basis ($1\in\mathbb{K}$) of the homology group at the birth of the longer bar. On the other hand both the equivalence classes of two smaller loops (right and left) can serve as a basis ($(0,1)\in\mathbb{K}\oplus\mathbb{K}$) of the homology group at the birth of the shorter bar.
\begin{figure}[ht]
\centering
\includegraphics[width=\linewidth]{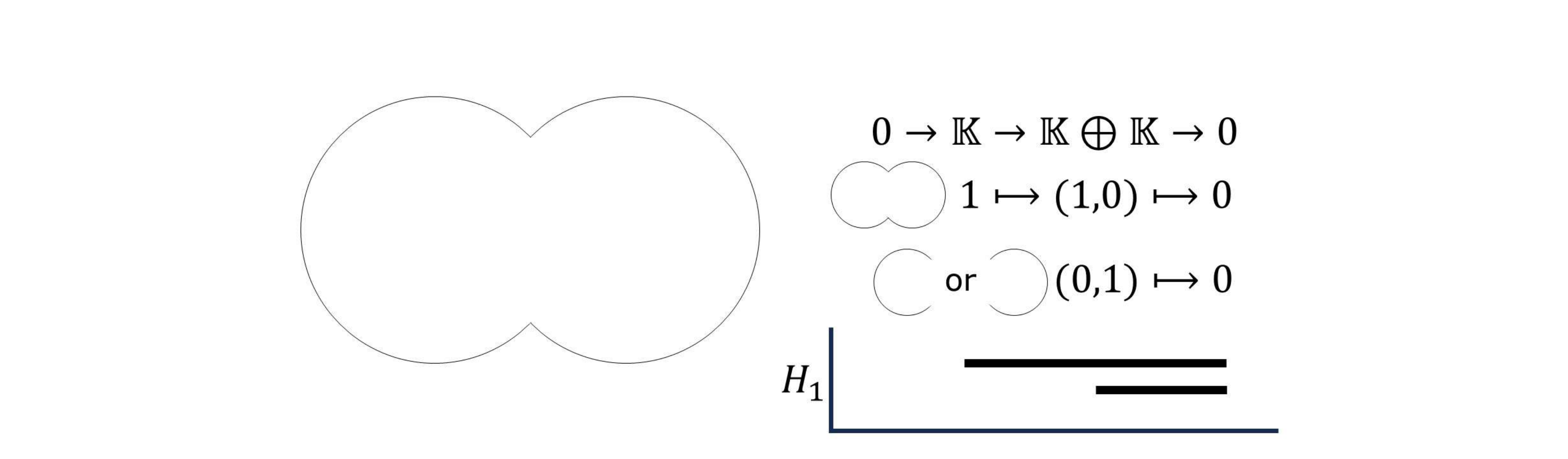}
\caption{An example to make clear the meaning of the expression: a loop's equivalence class can serve as a basis of the homology group at the birth of a bar}
\label{barcodechoice}
\end{figure}
\begin{dfn}[detectable internal structure]
A loop $L$ is called a detectable internal structure at $U$ if there exists a point $p\in PD_1(X)\cap PD_1(X\cup U)$ such that $L$ can represent $p$ both for $PD_1(X)$ and $PD_1(X\cup U)$.
\end{dfn}
We now state the proposition which shows how the internal structure behaves with respect to the points on the convex hull.
\begin{proposition}
\label{intstrmono}
If a loop that is a detectable internal structure at $U=\emptyset$ is not a detectable internal structure at $U=U_1$, then it is not a detectable internal structure at $U=U_2$.
\end{proposition}
In the remaining part of this subsection, we will make some definitions and statements for the proof of Prop.~\ref{intstrmono} and prove it. We first define the birth and death for loops and linear combinations of loops in a similar way as in Oda(2023)\cite{oda}.
\begin{dfn}[birth and death for loops]
The birth of a loop or a linear combination of loops $L$ denoted by $B(L)$ is $b$ if all the 1-simplex in the loop or the linear combination of loops are in $C^b_1$ for the first time at $r=b$.
The death of a loop or a linear combination of loops $L$ denoted by $D(L)$ is $d$ if the loop or the linear combination of loops is in the image of the boundary map for the first time at $r=d$.
\end{dfn}
For a bar $Bar=[b,d]$, we denote the birth $b$ by $B(Bar)$ and death $d$ by $D(Bar)$.
\begin{dfn}
We say that a loop or a linear combination of loops $L$ can be in a bar $[b,d]$ if $B(L)\in[b,d)$ and $[L]$ can serve as a basis of the homology group of the bar at $r=B(L)$.
\end{dfn}
Note that if a loop $L$ can be in a bar $[b,d]$, then $D(L)=d$.
\begin{lemma}
\label{lemma}
Let $L, L'$ be two loops or linear combinations of loops. Suppose that $\exists k\in\mathbb{K}\backslash\{0\}, L\sim kL'$ at $r=c$. Let $D(L)=D(L')=d$. Let $B(L')= b'<c$. Then, for any $b$ such that $b'<b<d$, $L$ cannot be in a bar $[b,d]$. Also, $L$ cannot represent a bar $[b,d]$.
\end{lemma}
\begin{proof}
$L'$ can be written as the linear combination of loops whose equivalence classes are the bases of homology groups of bars at $r=b'$. After $b'$, the same linear combination of loops is equivalent to $L'$ at any radius. If all the bars' deaths are less than $d$, then $D(L')<d$, which is a contradiction. In the linear combination of the loops, we choose the ones whose death points are $d$ (this set is not empty). We set all the other loops' coefficients in the linear combination to be $0$. There exists a positive number $\epsilon$ such that the new linear combination is equivalent to $L'$ after $d-\epsilon$. This linear combination is also equivalent to $(1/k)L$ after $\max\{c,d-\epsilon\}$. Suppose that $L$ can be in a bar $[b,d]$. Since $b'<b$, no loop in the new linear combination is chosen from the bar $[b,d]$. This contradicts the fact that the bars are linearly independent at any point. We can use the same argument if we suppose that $L$ can represent a bar $[b,d]$.
\end{proof}
\begin{lemma}
\label{lemma2}
If a loop $L$ with $B(L)=b$ and $D(L)=d$ can represent a bar, then the bar is $[b,d]$.
\end{lemma}
\begin{proof}
Suppose that $L$ can represent a bar $[b',d]$. It is impossible to have $b'<b$. Thus, the only possibility for $b\neq b'$ is $b<b'(<d)$. $L\sim L$ at $r=b'$. $B(L)$ is less than $b'$. Then, from Lem.~\ref{lemma}, $L$ class represent a bar $[b',d]$.
\end{proof}
From this, we can see that if a loop can represent a bar, then the loop can be in the bar.
Now, we will prove Prop.~\ref{intstrmono}.
\begin{proof}[Proof of Prop.~\ref{intstrmono}]
Suppose that a loop $L$ with $B(L)=b$ and $D(L)=d$ is a detectable internal structure at $U=\emptyset$. From Lem.~\ref{lemma2}, $[b,d]$ is the only bar that $L$ can represent. Therefore, $L$ not being a detectable internal structure at $U=U_1$ means that at $U=U_1$, $L$ cannot represent the bar $[b,d]$. We divide the proof into cases.

\vspace{5pt}

\noindent\underline{Case1}: $L$ cannot be in a bar at $U=U_1$.\\
\hspace*{8pt}Case1-1: $B(L)=D(L)$.\\
\hspace*{8pt}Case1-2: $L$ can only be expressed as a linear combination of loops whose equivalence classes are bases of homology groups of multiple bars.\\
\underline{Case2}: $L$ can be in a bar at $U=U_1$.\\
\hspace*{8pt}Case2-1: The birth of the bar in Case2 must be smaller than $b$.\\
\hspace*{8pt}Case2-2: $D(L)\neq d$.

\vspace{5pt}

Note that when we increase the points in the point cloud, the birth of a loop does not change and the death of a loop does not get large. Therefore, we see that in Case 1-1 and Case 2-2, we get into the same situation at $U=U_2$.

When we deal with Case 1-2 and Case 2-1, we can assume that $D(L)=d$ at $U=U_1$ and $U_2$, because otherwise, if $D(L)$ is not $d$ at $U=U_1$, we can use a similar argument as in Case1-1 and Case2-2, and if $D(L)$ is not $d$ at $U=U_2$, from Lem.~\ref{lemma2}, $L$ cannot represent a bar $[b,d]$.

First, we look at Case 2-1. Let $b'$ be the birth of the bar that $L$ can be in. Let $L'$ be the loop that can represent the bar $[b',d]$. At $U=U_1$, $L\sim kL'$ at $r=b$, where $k\in\mathbb{K}\backslash\{0\}$. Also, we have $B(L')=b'<b$. These two properties also hold at $U=U_2$. Since $b'<b<d$, from Lem.~\ref{lemma}, $L$ cannot represent a bar $[b,d]$ at $U=U_2$.

Finally, we deal with Case 1-2. At $r=b$, take the loops $L_1,\ldots,L_k$ where $[L_i]$s are bases of homology groups of multiple bars $Bar_1,\ldots,Bar_k$ such that $L=c_1L_1+\cdots+c_kL_k$, where $c_1,\ldots,c_k$ are coefficients. Since $D(L)=d$, $\max_{1\le i\le k}(D(Bar_i))=d$. Define $I$ to be $I:=\{1\le i\le k|D(Bar_i)=d\}$. If there exists $i\in I$ such that $B(Bar_i)=b$, then we can choose $[L]$ as a basis of the homology group of $Bar_i$ instead of $[L_i]$, which is a contradiction. Thus, $b'':=\max_{i\in I}(B(Bar_i))<b$. At $r=b''$, take the loops $\{L''_i\}_{i\in I}$ where $[L''_i] (i\in I)$ are bases of homology groups of $Bar_i (i\in I)$. We can choose the coefficients $c''_i (i\in I)$ such that $[c''_iL''_i]=[c_iL_i]$ at $r=b$. Let $L''$ be $L'':=\sum_{i\in I}c''_iL''_i$. There exists a sufficiently small positive number $\epsilon$ such that $[L'']=[L]$ at $r=d-\epsilon$. We also have $B(L'')\leq b''<d-\epsilon$. These two properties also hold at $U=U_2$. Since $b''<b<d$, from Lem.~\ref{lemma}, $L$ cannot represent a bar $[b,d]$ at $U=U_2$.
\end{proof}

This result can be understood as the monotonicity relationship with respect to the plots on the convex hull.
\subsection{External structures}
Now, we deal with external structures. As we defined earlier, at $U=U_1$ and $U_2$, the detected internal structures are $PD_1(X\cup U_1)\backslash PD_1(X)$ and $PD_1(X\cup U_2)\backslash PD_1(X)$, respectively.
\begin{dfn}
A loop $L$ is called a detectable external structure at $U$ if there exists a point $p\in PD_1(X\cup U)\backslash PD_1(X)$ such that $L$ can represent $p$ for $PD_1(X\cup U)$.
\end{dfn}
For external structures, monotonicity relationships as in the internal structures do not hold. This can be understood from Fig.~\ref{branchtheory}. In Fig.~\ref{branchtheory}(b), we add a point on the left side of the image in Fig.~\ref{branchtheory}(a). In Fig.~\ref{branchtheory}(c), we add points on the lower part of the image in addition to the point on the left side. We calculate persistent homology for these three images. The barcodes are shown in a way that makes it clear how the bases of the homology groups in the bars are mapped when the radius changes. We will see that the large loop which is a detectable external structure in Fig.~\ref{branchtheory}(b) is not a detectable external structure in Fig.~\ref{branchtheory}(c).
\begin{figure}[ht]
\centering
\includegraphics[scale=0.2]{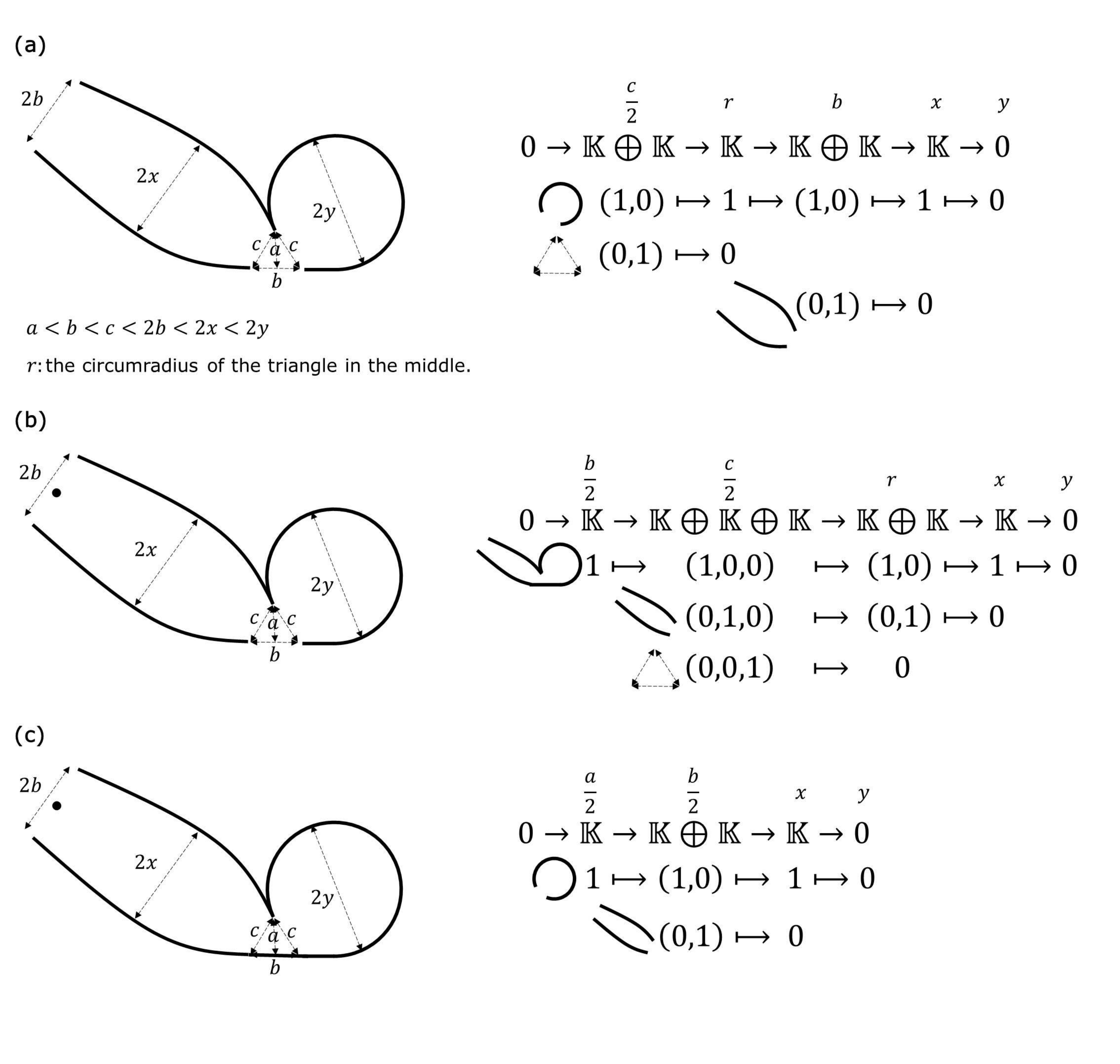}
\caption{(a) The original image and its barcode. (b) The original image plus the point on the left opening and the resulting barcode. (c) The image in (b) plus the points on the lower opening and the resulting barcode. The large loop which is a detectable external structure in (b) cannot be a detectable external structure in (c)}
\label{branchtheory}
\end{figure}
This phenomenon is related to the resolution of the image. If we look at the original figure from far away, we only find the opening in the left part of the image. Then, we will think that the right part of the figure is part of a large external structure. On the other hand, if we take a close look at the image, we will also find the opening in the lower part of the image.
\section{Results}
In the following part of the paper, the unit for the lengths is pixel unless otherwise specified.
\subsection{Lymph vessels}
We first count the number of external/internal structures in the images of lymph vessels. We plot points on the convex hull at the interval of 20. We count the number of detected structures whose persistences are larger than 5. The result is shown in Fig.~\ref{branchnumber}. We plotted the circumcenter of the death positions as in Oda et al. (2023) \cite{oda2}.
\begin{figure}[h]
\centering
\includegraphics[width=\linewidth]{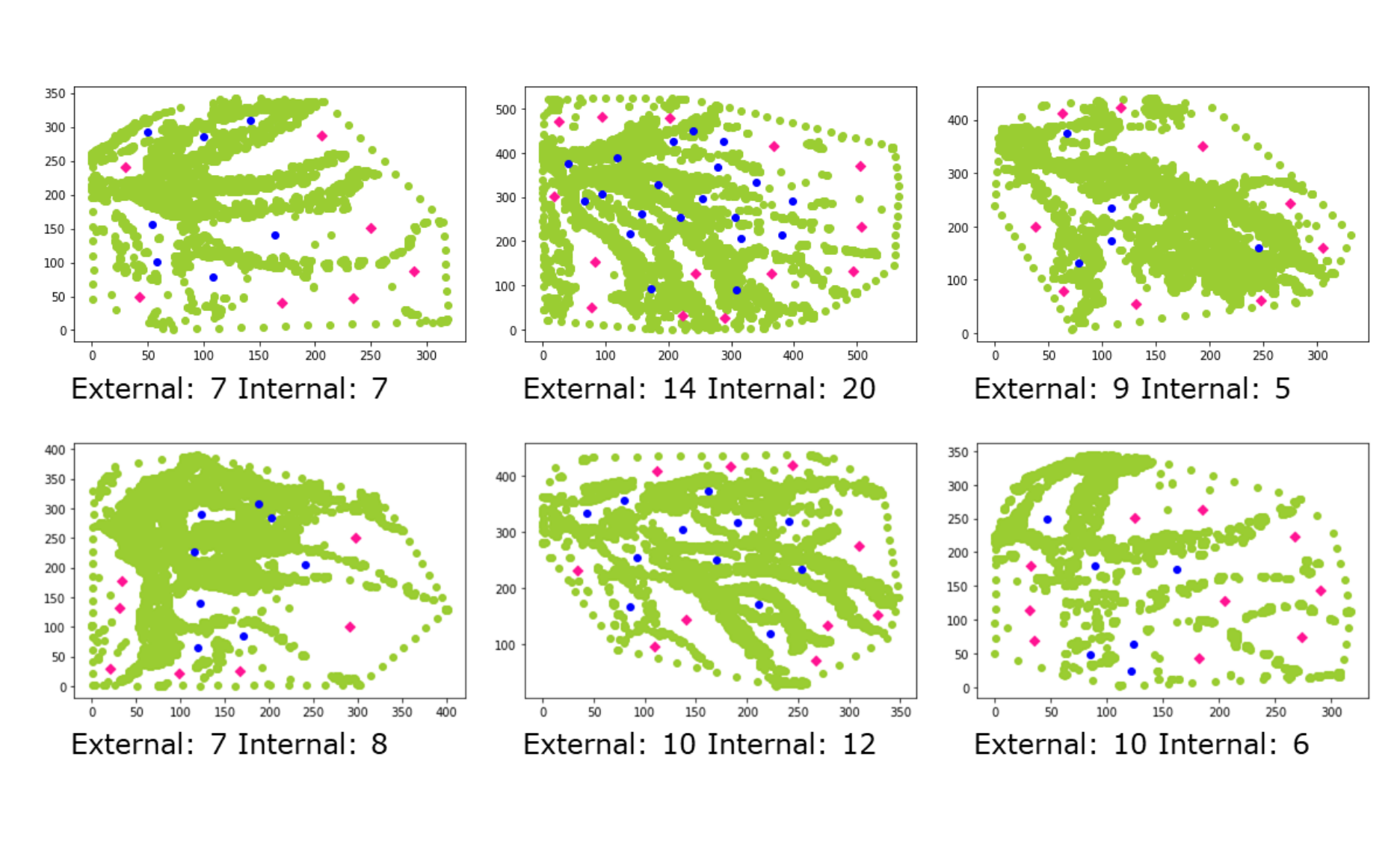}
\caption{The result of detecting external/internal structures in the images (red/blue points)}
\label{branchnumber}
\end{figure}
Since we have the information of positions of external/internal structures, we can explore their spatial distributions. We can also extract their size information from death points, which can be used, for example, for the comparison of external and internal structure sizes.

Since we have constructed diagrams for external and internal structures, we can also conduct more topological analyses of our results. For example, we can focus on the internal structures and construct persistence landscapes from them. We calculate the integral of the square of the first function in the persistence landscape. The result is shown in Online Resource (Supplementary Fig. S1). This could be understood as a measure of some kind of internal complexity. Constructing these kinds of measures from diagrams and finding biologically important ones will be interesting future work (See ``Discussion'' section).
\subsection{Neurons}
Our method can be applied to a wide range of branch structures that appear in biology. Here, we present an example of the analysis of branch structure in neurons. Figure~\ref{neuron}(a) shows the result of analyzing the image in Tschida et al. (2012)\cite{tschida} (Fig. 1 A, we used the image without the white arrow that is available in \texttt{https://medicalxpress.com/news/2012-03-deafening-affects-vocal-nerve-cells.html}.) The interval for the plots on the convex hull was 10, and the threshold for the persistences was 2. We also analyzed the image in Desai-Chowdhry et al. (2022)\cite{desai-chowdhry} (Fig. 1). The interval for the plots on the convex hull was 15. If we set the threshold for the persistences to be 2, we get Fig.~\ref{neuron}(b). If we want to detect only the long persisting internal structures, we can increase the threshold of the persistence for the internal structure. As a result, we get Fig.~\ref{neuron}(c) (the threshold for internal structure is changed into 7).
\begin{figure}[h]
\centering
\includegraphics[width=\linewidth]{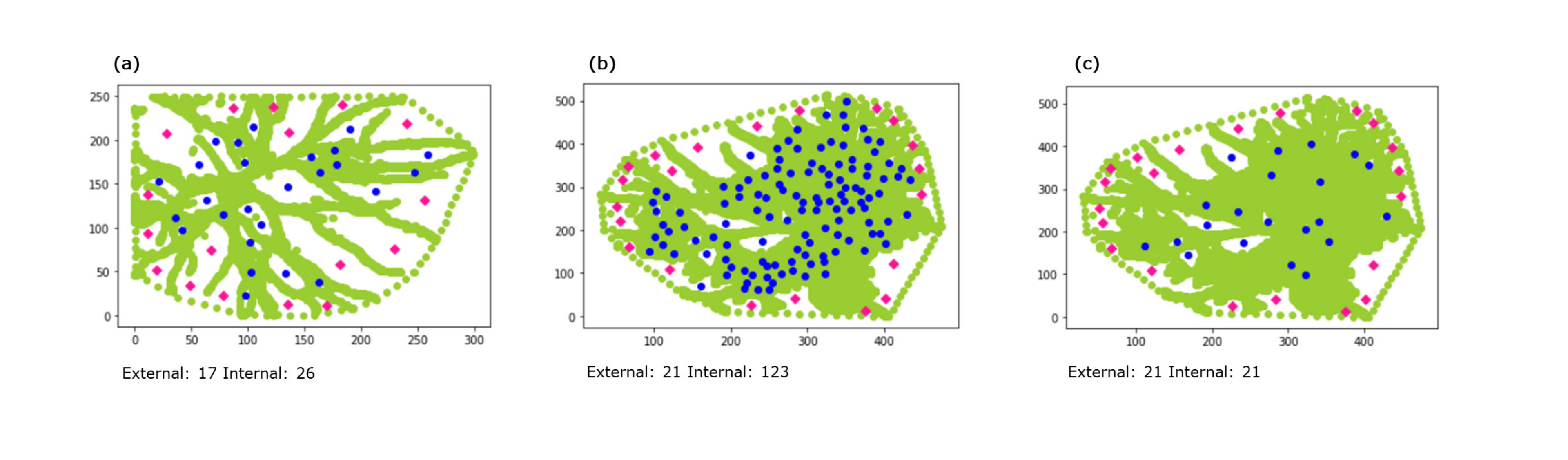}
\caption{(a) The result of analyzing the image of a neuron\cite{tschida}. The detected external/internal structures are plotted by red/blue points. (b,c)The result of analyzing the image of a neuron\cite{desai-chowdhry}. We can increase the threshold of persistence if we want to focus on the long persisting structures}
\label{neuron}
\end{figure}
\subsection{Blood vessels}
Now, we will apply our method to blood vessels in Liu et al. (2010)\cite{liu} (Fig. 2B left and right). Figure~\ref{bloodvessel}(a) shows the result of analyzing the two images with the interval for the plots on the convex hull being 15 and the threshold of the persistences being 2. We detect more internal structures in the left image (Notch3(+/+)) compared to the right image (Notch3(-/-)). If we change the threshold of the persistence for internal structure to 10, we detect fewer internal structures in Notch3(+/+) than in Notch3(-/-) (Fig.~\ref{bloodvessel}(b)).
\begin{figure}[h]
\centering
\includegraphics[width=\linewidth]{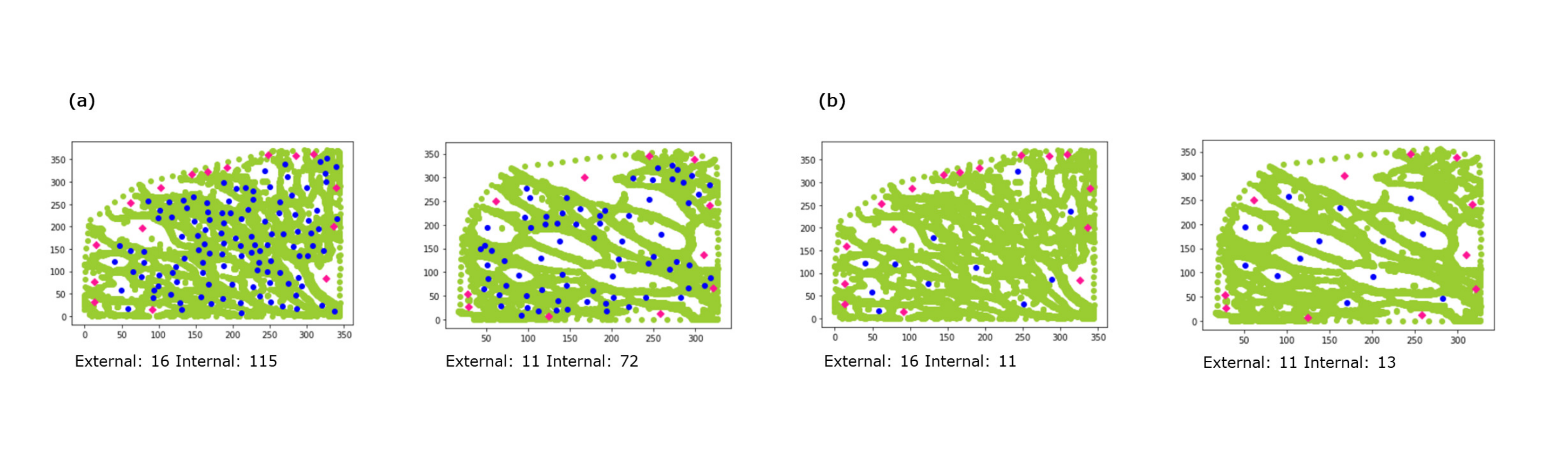}
\caption{(a) The result of analyzing the images in Liu et al.\cite{liu}. More internal structures are detected in the left image compared to the right image. (b) The result of increasing the threshold of the persistences for the internal structures. Now, we detect more internal structures in the right image compared to the left image}
\label{bloodvessel}
\end{figure}
\section{Discussion}
\subsection{Comparison with non-TDA methods}
In Thunbo et al. (2017)\cite{placentaarea}, they analyze the vessel using the area of white points divided by the area of the convex hull. This method is reasonable because they tried to analyze the density of the vessels. It is expected that this method returns a larger value if we make the width of the vessels larger without changing their shape. On the other hand, we aim to look at external and internal structures. Thus, our method should not be affected greatly by the change in the width of the vessels. We use model images in Online Resource (Supplementary Fig. S2) to see this difference. Table~\ref{comparison} shows that the method in Thunbo et al. (2017) is affected by the change in the width of vessels, while our method is not. We set the interval of the plots on the convex hull to be $3$ and the threshold for the persistence to be $0.5$.
\begin{table}[ht]
\centering
\begin{tabular}{ccccccccc}
\hline
method & model 1&model 2&model 3&model 4&model 5&model 6&model 7&model 8 \\
\hline\hline
\begin{tabular}{c}
vessel area\\over convex hull
\end{tabular}&0.13&0.18&0.24&0.29&0.34&0.38&0.42&0.46\\
\hline
\begin{tabular}{c}
number of external\\/internal structure\\determined with PH
\end{tabular}&10/11&9/11&10/11&9/11&11/11&10/10&10/12&10/12\\
\hline
\end{tabular}
\caption{\label{comparison}The method in Thunbo et al. (2017) is affected by the change in the width of vessels, while our method is not. A model with a larger number has a greater line width}
\end{table}

Skeletonization is one of the standard methods to analyze branch or network-like structures as in Carpentier et al. (2020)\cite{imageJskeleton}, where they analyze angiogenesis. One of our purposes is to determine whether some structure is a branch or not, but skeletonization makes everything into branch-like images. Thus, this is not suited to our purpose.

Since the main part of our method is to define external and internal structures, there is no such thing as ``ground truth'' or ``expected result''. Thus, it is impossible to give some statistical evaluation of our results against them. Rather, it might be interesting as future works to create several features based on our method and use them for images, for example, created by a simulation of branch formation in order to find features that express some kind of biological status.

\subsection{Comparison with other TDA methods}
TDA has been applied to analyze various biological structural patterns. When we focus on branch structures, Reimann et al. (2017)\cite{reimann} and Kanari et al. (2018, 2019)\cite{kanari,kanari3} deal with neuronal structures. Bendich et al. (2016)\cite{bendich} analyzes brain arteries and Jiang (2022)\cite{jiang} analyzes retina blood vessels using persistent homology. When we look for a wider range of patterns in biology or medicine, Tanabe et al. (2021)\cite{tanabe} uses TDA for the CT images of COPD, IPF, and CPFE, and Giansiracusa et al. (2019)\cite{giansiracusa} conducts fingerprint classification with persistent homology. Thomas et al. (2021)\cite{thomas} takes a topological approach to \textit{C. elegans} movement patterns. Also, there are various related theoretical results such as Kanari et al. (2020)\cite{kanari2}.

Our method adds to this existing TDA knowledge especially in the following points:
\begin{itemize}
\item{We intentionally add points on the convex hull to define the internal and external structures. Adding new points intentionally to the original data to detect some desired structures is a novel way to use TDA.}
\item{We compare two different persistence diagrams and define new diagrams by taking the intersection or difference between them. We also constructed mathematical theories for these new diagrams. This, combined with the above point, provides a novel framework for TDA itself as well as its new usage.}
\end{itemize}

\subsection{More discussions on the underlying idea of our framework}
In the theory section, we explained how the internal and external structures change under the change of the points on the convex hull.
Now, the readers might wonder whether the results of our framework (internal/external structures) are extracting some intrinsic property of the original data. This question is related to the underlying idea of our framework: the notion of internal/external structures is not determined by the data alone but also by our purposes. For example, we might want to limit our focus on the ``very best" internal structures or we might want to include some ``intermediate" structures.

To justify these subjective ideas, an objective and quantitative description should be developed. Using our framework, we can explain the criteria used for the internal/external structure distinction objectively and quantitatively, e.g. ``We plotted points on the convex hull of the images at the interval of 20."
In Online Resource (Supplementary Fig. S3), we show an example of the change of internal/external structures with respect to the plots on the convex hull.

Our framework naturally connects to multiparameter persistence, and once its data analysis framework (such as mapping the algebraic structures back to the original data) is developed, it will improve our framework.


\subsection{Generalized persistence landscape}
Finally, we will introduce the generalized persistence landscape. Although we do not use it in this paper, we will construct it for future applications. One of the good points of the persistence landscape is that the process is invertible. We can recover the persistence diagram from the persistence landscape. The problem with persistence landscape is that, since it uses the diagonal line as the baseline, it assumes the importance of the lengths of bars.
Sometimes, the important features of the data might not lie in long bars\cite{bendich}. Also, the fact that our method preserves the information on births and so on allows us to expect that our method might work well with a tool that makes it possible to focus on features other than persistence.
We want to generalize the persistence landscape in a way that meets this requirement and at the same time preserves the invertibility.
We define the generalized persistence landscape as follows. Let $R(\phi): \mathbb{R}^2\rightarrow\mathbb{R}^2$ denote the rotation by $\phi$ around the origin, and $T(d): \mathbb{R}^2\rightarrow\mathbb{R}^2$ denote the translation by $d$ along the $y$-axis. For a persistence diagram $D$, we denote by $\lambda(D)$ the persistence landscape constructed from $D$. For $X\subset\mathbb{R}^2$, $X_{+}:=\{(x,y)\in X|x\le y\}$ and $X_{-}:=\{(x,y)\in X|x>y\}$. Let $S: \mathbb{R}^2\rightarrow\mathbb{R}^2$ denote the swapping of the coordinates ($(x,y)\mapsto (y,x)$).
\begin{dfn}[Generalized persistence landscape]
For a persistence diagram $D\subset\mathbb{R}^2$, the generalized persistence landscape $G\lambda[\theta,y](D)$ is defined to be $G\lambda[\theta,y](D):=\lambda((R(\pi/4-\theta)\circ T(-y)(D))_{+})\cup \lambda(S((R(\pi/4-\theta)\circ T(-y)(D))_{-}))$. We represent the elements of $\lambda((R(\pi/4-\theta)\circ T(-y)(D))_{+})$ as $G\lambda_1, G\lambda_2,\ldots,G\lambda_k,\ldots$ and the elements of $\lambda(S((R(\pi/4-\theta)\circ T(-y)(D))_{-}))$ as $G\lambda_{-1}, G\lambda_{-2},\ldots,G\lambda_{-k},\ldots$.
\end{dfn}
This definition means that we change our baseline from the diagonal line to an arbitrary line. See the example figure in Online Resource (Supplementary Fig. S4).


Since $\lambda$, $R(\phi)$, and $T(d)$ are all invertible, we can recover the persistence diagram from the generalized persistence landscape. The original persistence landscape $\lambda$ is included in the generalized persistence landscape as $G\lambda[\pi/4,0]$.

The concepts of linear combination, difference, and norm of persistence landscapes can be extended to generalized persistence landscapes. The only difference is that now we have functions with negative subscripts.

When we want to use the generalized persistence landscape for image classification, we, for example, calculate the generalized persistence landscape for each of the images, take the average landscape for each of the groups, and calculate the distance of a landscape from the average landscapes. We classify the landscape to the group whose average landscape is closest to the input landscape. We change the parameters and see how the classification result changes. This process goes along well with gradient descent in machine learning because the distance between the landscape and the average landscape is partially differentiable with respect to $\theta$ and $y$. We will prove this in a more general setting. We use 2-norm for the distance. In the following proposition, the rotation and translation are generalized as moving points of the diagrams. Thus, the points in persistence diagrams can be in the region $y\le x$, and when we say persistence landscape, we mean $G\lambda[\pi/4,0]$.
\begin{proposition}
\label{prop1}
Suppose that we are given a persistence diagram $D$ and persistence diagrams ${D'}^j (j=1,\ldots,N)$. For each of the points $p_t$ in the persistence diagrams, we give constants $\alpha_t, \beta_t$. We move each of the points $p_t$ in the persistence diagrams by $\alpha_t\epsilon+O(\epsilon^2)$ along the $x$-axis and by $\beta_t\epsilon+O(\epsilon^2)$ along the $y$-axis. Let $\tilde{D}$ and $\tilde{{D'}^j}$ be the resulting diagrams. Define $d(\epsilon)$ to be the distance between the landscape of $\tilde{D}$ and the average landscape of $\tilde{{D'}^j}$. Then, $d(\epsilon)$ is differentiable at $\epsilon=0$.
\end{proposition}
\begin{proof}
Let $\lambda_i (i=-k,\ldots,-1,1,\ldots,n)$ be the elements of persistence landscape for $D$, and ${\lambda'_i}^j$ be the elements of persistence landscape for ${D'}^j$. Let $\tilde{\lambda_i}$ and $\tilde{{\lambda'_i}^j}$ be the elements of corresponding persistence landscapes after moving points.
$d(0)$ can be written as
\[
d(0)=\frac{1}{N}\sqrt{\sum_{i=1}^n\int(\sum_{j=1}^N\lambda_i-{\lambda'_i}^j)^2dx+\sum_{i=-k}^{-1}\int(\sum_{j=1}^N\lambda_i-{\lambda'_i}^j)^2dx}
\]
$d(\epsilon)$ can be written similarly with tildes on lambdas.
Choose $\epsilon$ to be so small that for all $\lambda_i$ or ${\lambda'_i}^j$ and corresponding $\tilde{\lambda_i}$ or $\tilde{{\lambda'_i}^j}$, the difference $\epsilon_i=\lambda_i-\tilde{\lambda_i}$ or ${\epsilon'}_i^j={\lambda'_i}^j-\tilde{{\lambda'_i}^j}$ satisfies the following: for all intervals where $\lambda_i$ or ${\lambda'_i}^j$ does not bend, the slope of $\epsilon_i$ or ${\epsilon'}_i^j$ is $0$ except for both ends whose width can be written as $b\epsilon$, where $b$ is the constant determined by $\alpha_t, \beta_t$. The value of $\epsilon_i$ or ${\epsilon'}_i^j$ where its slope is $0$ can be written as $c\epsilon$, where $c$ is the constant determined by $\alpha_t, \beta_t$. We need to show that $(d(\epsilon)-d(0))/\epsilon$ converges when $\epsilon\to 0$. Since we have the following equality:
\begin{equation*}
\begin{split}
\int(\sum_{j=1}^N\tilde{\lambda_i}-\tilde{{\lambda'_i}^j})^2dx=&\int(\sum_{j=1}^N\lambda_i-{\lambda'_i}^j)^2dx+2\sum_{j=1}^N\sum_{s=1}^N\int(\lambda_i-{\lambda'_i}^j)(\epsilon_i-{\epsilon'_i}^s)dx\\
&+\sum_{j=1}^N\sum_{s=1}^N\int(\epsilon_i-{\epsilon'_i}^j)(\epsilon_i-{\epsilon'_i}^s)dx,
\end{split}
\end{equation*}
we can prove the statement by showing that the last 2 terms in the above equality are of the form $const.\epsilon+O(\epsilon^2)$.

We first deal with $\int(\epsilon_i-{\epsilon'_i}^j)(\epsilon_i-{\epsilon'_i}^s)dx$. We have $|\int(\epsilon_i-{\epsilon'_i}^j)(\epsilon_i-{\epsilon'_i}^s)dx|\le\int|\epsilon_i-{\epsilon'_i}^j||\epsilon_i-{\epsilon'_i}^s|dx$. The maximum value of $|\epsilon_i-{\epsilon'_i}^j|$ can be written as $r\epsilon$, where $r$ is a constant determined by $\alpha_t,\beta_t$. Also, the support of $|\epsilon_i-{\epsilon'_i}^j||\epsilon_i-{\epsilon'_i}^s|$ has finite length. Therefore, we have $|\int(\epsilon_i-{\epsilon'_i}^j)(\epsilon_i-{\epsilon'_i}^s)dx|\le O(\epsilon^2)$.

Next, we deal with $\int(\lambda_i-{\lambda'_i}^j)(\epsilon_i-{\epsilon'_i}^s)dx$. We calculate this integral for each of the intervals where $\lambda_i$ and ${\lambda'_i}^j$ do not bend nor does their difference change signs. Take one of such intervals $I$. On $I$, both $\lambda_i$ and ${\lambda'_i}^j$ do not bend, which means that $\epsilon_i-{\epsilon'_i}^s$ has slope $0$ except for the both ends. The width of these ends and the maximum absolute value of $\epsilon_i-{\epsilon'_i}^s$ at these ends can be bounded by $u\epsilon$ and $v\epsilon$, where $u,v$ are constants independent of $\epsilon$. The absolute value of $\lambda_i-{\lambda'_i}^j$ on $I$ is bounded by a constant $l$, which is independent of $\epsilon$. The value of $\epsilon_i-{\epsilon'_i}^s$ where its slope is $0$ can be written as $w\epsilon$, where $w$ is a constant independent of $\epsilon$. Therefore, we can write
\[
\int(\lambda_i-{\lambda'_i}^j)(\epsilon_i-{\epsilon'_i}^s)dx=w\epsilon\int(\lambda_i-{\lambda'_i}^j)dx+\delta
\]
where $\delta$ is the error term bounded by $uvl\epsilon^2=O(\epsilon^2)$. Note that the interval of the integral $I$ is not affected by $\epsilon$, which means that the first term on the right-hand side of the above equality is of the form $const.\epsilon$.
\end{proof}
From Prop.~\ref{prop1}, we get the following corollary.
\begin{corollary}
\label{cor2}
Suppose that we are given a persistence diagram $D$ and persistence diagrams ${D'}^j (j=1,\ldots,N)$. The distance between the generalized persistence landscape $G\lambda[\theta,y](D)$ and the average of the generalized persistence landscapes $G\lambda[\theta,y]({D'}^j)$ is partially differentiable with respect to $\theta$ and $y$.
\end{corollary}

\section{Conclusion}
In this paper, we constructed the persistent homological definitions of external and internal structures for discrete branch-like structures. Based on this definition, we developed a method to detect external and internal structures in the images. We also established mathematical theories explaining the behavior of external and internal structures with respect to the points on the convex hull. We showed that our method can detect external and internal structures in a way that can be used for numerical analysis. As examples, we counted the number of external/internal structures and also conducted an internal structure analysis with persistence landscape. We showed that our method can be applied to a wide range of branch structures that appear in biology. We also explained the advantages of our method over other methods using TDA. For future applications, we constructed the generalized persistence landscape.

With these tools, we can develop various quantitative measures of external/internal structures and search for the ones with biological significance in specific contexts in future works. As for applications, we can use our method for various ``branch-like'' structures to objectively analyze and compare them.

\section*{Statements and Declarations}
\subsection*{Data Availability}
The datasets analyzed during the current study are available from the corresponding author upon reasonable request.
\subsection*{Code Availability}
The program code and model images can be found in the following link. \url{https://drive.google.com/drive/folders/1sjPd2z\_63iZEEInYAETRZKqUualA3plt?usp=sharing}
\subsection*{Competing Interests}
None.

\section*{Author contributions statement}
H.O., M.K., and H.K. conceived the research idea. H.O. developed the lymph vessel analysis method based on persistent homology, generalized persistence landscape, and their theoretical backgrounds and drafted the manuscript. M.K. took images of lymph vessels. Y.N. proofread the original manuscript, fixed some ambiguous expressions, and introduced several symbols to make the explanations clear. H.O., M.K., and H.K. discussed the biological aspects of the method. H.O. and Y.N. discussed the mathematical aspects of the method. H.K. supervised the experimental results.
All authors reviewed the manuscript.


\end{document}